\documentclass[a4paper,10pt,reqno]{amsart}

\usepackage[utf8]{inputenc}
\usepackage[T1]{fontenc}
\usepackage{amsthm}
\usepackage{amsmath}
\usepackage{amssymb}
\usepackage[inline]{enumitem}
\usepackage{comment}
\usepackage{bbold}
\PassOptionsToPackage{hyphens}{url}\usepackage{hyperref}
\usepackage{fancyhdr}
\usepackage{mathrsfs}
\usepackage{stmaryrd}
\usepackage{tikz-cd}
\usetikzlibrary{arrows}
\usepackage[usestackEOL]{stackengine}
\usepackage{scalerel}
\usepackage{array}
\usepackage[all,cmtip]{xy}
\usepackage{graphicx}
\usepackage{yhmath}
\usepackage[normalem]{ulem}

\theoremstyle{definition}

\newtheorem{theorem}{Theorem}[section]
\newtheorem{lemma}[theorem]{Lemma}
\newtheorem{proposition}[theorem]{Proposition}

\theoremstyle{definition}
\newtheorem{definition}[theorem]{Definition}
\newtheorem{example}[theorem]{Example}
\newtheorem{remark}[theorem]{Remark}

\definecolor{blue-url}{RGB}{0,0,100}
\definecolor{red-url}{RGB}{100,0,0}
\definecolor{green-url}{RGB}{0,100,0}
\definecolor{light-yellow}{RGB}{255,255,128}
\definecolor{light-blue}{RGB}{193,255,255}
\definecolor{light-red}{RGB}{239,83,80}

\hypersetup{
	pdftitle={UmF-ness and FmF-ness in power monoids},
	pdfauthor={},
	pdfmenubar=false,
	pdffitwindow=true,
	pdfstartview=FitH,
	colorlinks=true,
	linkcolor=blue-url,
	citecolor=green-url,
	urlcolor=red-url
}

\renewcommand{\emptyset}{\varnothing}
\renewcommand{\setminus}{\smallsetminus}
\renewcommand{\,}{\kern 0.1em}

\providecommand\llb{\llbracket}
\providecommand\rrb{\rrbracket}

\providecommand{\RR}{\mathbin{\Cal R}}

\newcommand{\Mon}{\mathsf{Mon}}
\newcommand{\FaMon}{\mathsf{FaMon}}
\newcommand{\AtoMon}{\mathsf{AtoMon}}
\newcommand{\evid}[1]{\textsf{#1}}

\newcommand{\HH}{\mathsf{C}_{\mathcal H}}
\newcommand{\gword}{(a_1,i_1)\ast (a_2,i_2) \ast \cdots \ast (a_n,i_n)} 
\newcommand{\FC}{\mathscr{F}(\sqcup\, \mathcal H)} 
\newcommand{\Cal}[1]{\mathcal{#1}}
\newcommand{\PA}{\mathsf{P}_\mathcal{H}} 
\newcommand{\UPA}{\mathscr{U}_\Cal H} 
\newcommand{\APA}{\mathscr{A}_\Cal H} 
\newcommand{\pullbackcorner}[1][dr]{\save*!/#1-1.2pc/#1:(-1,1)@^{|-}\restore} 

\makeatletter
\DeclareFontFamily{OMX}{MnSymbolE}{}
\DeclareSymbolFont{MnLargeSymbols}{OMX}{MnSymbolE}{m}{n}
\SetSymbolFont{MnLargeSymbols}{bold}{OMX}{MnSymbolE}{b}{n}
\DeclareFontShape{OMX}{MnSymbolE}{m}{n}{
	<-6>  MnSymbolE5
	<6-7>  MnSymbolE6
	<7-8>  MnSymbolE7
	<8-9>  MnSymbolE8
	<9-10> MnSymbolE9
	<10-12> MnSymbolE10
	<12->   MnSymbolE12
}{}
\DeclareFontShape{OMX}{MnSymbolE}{b}{n}{
	<-6>  MnSymbolE-Bold5
	<6-7>  MnSymbolE-Bold6
	<7-8>  MnSymbolE-Bold7
	<8-9>  MnSymbolE-Bold8
	<9-10> MnSymbolE-Bold9
	<10-12> MnSymbolE-Bold10
	<12->   MnSymbolE-Bold12
}{}

\let\llangle\@undefined
\let\rrangle\@undefined
\DeclareMathDelimiter{\llangle}{\mathopen}%
{MnLargeSymbols}{'164}{MnLargeSymbols}{'164}
\DeclareMathDelimiter{\rrangle}{\mathclose}%
{MnLargeSymbols}{'171}{MnLargeSymbols}{'171}
\makeatother

\begin{document}

\title{The Category of Atomic Monoids: \\ Universal Constructions and Arithmetic Properties}

\author{Federico Campanini}
\address{(F.~Campanini) Institut de Recherche en Math\'ematique et Physique, Universit\'e Catholique de Louvain | Chemin du cyclotron 2, 1348 Louvain-la-Neuve, Belgium}
\email{federico.campanini@uclouvain.be}
\urladdr{https://uclouvain.be/fr/repertoires/federico.campanini}

\author{Laura Cossu}
\address{(L.~Cossu) Dipartimento di Matematica e Informatica, Università degli Studi di Cagliari | Palazzo delle Scienze, Via Ospedale 72, 09124 Cagliari (Italy)}
\email{laura.cossu3@unica.it}
\urladdr{https://sites.google.com/view/laura-cossu/}

\author{Salvatore Tringali}
\address{(S.~Tringali) School of Mathematical Sciences, Hebei Normal University | Shijiazhuang, Hebei province, 050024 China}
\email{salvo.tringali@gmail.com}
\urladdr{http://imsc.uni-graz.at/tringali}


\begin{abstract}
We introduce and investigate the category $\mathsf{AtoMon}$ of atomic monoids and atom-preserving monoid homomorphisms, which is a (non-full) subcategory of the usual category of monoids. 
In particular, we compute all limits and colimits, showing that $\mathsf{AtoMon}$ is a complete and cocomplete category. We also address certain arithmetic properties of products and coproducts, providing explicit formulas for some fundamental invariants associated with factorization lengths in atomic monoids.
\end{abstract}

\maketitle

\section{Introduction}

By the fundamental theorem of arithmetic, every integer greater than one is a product of primes in an essentially unique fashion. Going back to Euclid's early work on elementary number theory, this classical result has been greatly generalized over the centuries. One of the most basic ideas is been to replace the positive integers with a (multiplicatively written) monoid $H$ (see Sect.~\ref{sec:2.2} for some preliminaries on monoids). Accordingly, the role played by the primes is taken over by the atoms of $H$, where an \evid{atom} is a non-unit that does not factor as a product of two non-units. It is then possible to prove that, under certain conditions, every non-unit of $H$ factors as a (non-empty, finite) product of atoms, which is concisely expressed by saying that $H$ is an \evid{atomic} monoid and each of its non-units has an \evid{atomic factorization}. In general, such factorizations are far from being essentially unique in any sensible way (most notably, they need not be unique up to associates and permutations). In the past, this lack of uniqueness was viewed as a pathology (for example, in connection with the rise of algebraic number theory and the study of unique factorization in number rings). Over time, perspectives have, however, greatly changed and atomic factorizations, \textit{especially} when non-unique, are nowadays recognized as the trigger of a variety of interesting phenomena whose investigation has eventually led to the birth of  (\textit{classical}) factorization theory \cite{GeHK06}, where the focus is on commutative domains and cancellative commutative monoids.

Factorization theory is a subfield of algebra at the crossroad of various subjects, including ring theory, semigroup (and group) theory, and combinatorics. Its origins can be traced back to the late 1960s, when Narkiewicz initiated a systematic investigation of non-unique factorization phenomena in number rings. This study expanded into commutative algebra in the 1990s, 
eventually leading to the publication of Geroldinger and Halter-Koch's monograph \cite{GeHK06}. 
Over the past 15 years, there has been significant progress towards the extension of theory to the non-commutative cancellative setting \cite{Be-He-Le17, B-B-N-S23a, Ba-Sm15, Sm-Vo19a} and further \cite{Fa02, Fa-Wi04, Cam-Fac2018, An-Tr-2021, Cos-Tri-2021, Tr21(b)}.
In the meantime, the classical theory has also drawn new impetus from the work of Chapman, Coykendall, Gotti, and others, particularly in relation to monoid algebras \cite{C-C-G-S21, Co-Ha18, Co-Go19, Go18a,Go19a,Got-Li-2023, Ger-Got-2025}. 

In the present paper, we embrace a categorical approach to the study of factorization (see Sect.~\ref{sec:cats} for a quick review of some elementary aspects of category theory). More in detail, let $\Mon$ be the usual category of monoids and monoid homomorphisms, and let $\AtoMon$ be the subcategory of $\Mon$ whose objects are the \evid{atomic} monoids (that is, the monoids in which every non-unit factors as a finite product of atoms) and whose morphisms are the monoid homomorphisms $f \colon H \to K$ that map an atom of $H$ to an atom of $K$ (namely, the \evid{atom-preserving} morphisms of $\Mon$). 

In Section~\ref{sec:basics-on-atomon}, after introducing some fundamental properties of atomic monoids and atom-preserving monoid homomorphisms, we obtain a few basic structural results on the category $\AtoMon$. 

Section~\ref{sec: coproducts} focuses on coproducts. Besides establishing their existence, we show in Theorem~\ref{thm:non-empty-coproducts} that a canonical representative of the coproduct of a family of atomic monoids in $\AtoMon$ is given by their free product. This requires some non-trivial work and is a key step towards a more general principle, as we find in Sect.~\ref{sec: limits and colimits}
that all colimits in $\AtoMon$ exist and are computed in the same way as in $\Mon$.

Products, discussed in Section~\ref{sec: products}, are quite a different case.
In addition to their existence, we demonstrate in Theorem~\ref{thm:non-empty-products} that a canonical representative of the product object of a family $\{H_i\}_{i\in I}$ of atomic monoids, indexed by a non-empty set $I$, is not the direct product of the $H_i$'s (as in $\Mon$), but rather the submonoid of the direct product generated by the $I$-tuples $(x_i)_{i \in I}$ such that either all components are units (in their respective monoids), or all components are atoms.

Lastly, in Section \ref{sec: limits and colimits}, we construct  equalizers and coequalizers in $\AtoMon$, thus proving that the category is complete (Theorem \ref{thm: complete}) and cocomplete (Theorem \ref{thm: cocomplete}).

The existence and (explicit) description of limits and colimits in $\AtoMon$ are entirely new even from the perspective of the classical theory of factorization, and we expect these constructions to provide a novel and powerful tool for addressing some central problems in the field. In particular, products and coproducts, with the high degree of flexibility inherent in their construction, have great potential in addressing certain realization problems for length sets and closely related invariants (e.g., see \cite[Problem B]{Ge-Zh20},  \cite[Theorem 1.1]{Ge-Zh22}, and \cite[Theorem 1.1]{Ger-Got-2025}). First steps in this direction are undertaken in Theorems \ref{thm: s.s.o.l. coproduct} and \ref{thm: unions of s.o.l. coproduct} for coproducts, and in Theorems \ref{thm: s.s.o.l. product} and \ref{thm: unions product}
for products.

\section{Preliminaries}
\label{sec:2}
In this section, we establish notation and terminology used all through the paper. Further notation and terminology, if not explained when first introduced, are standard or should be clear from context.

\subsection{Generalities}\label{sec:2.1}
We assume that all relations are binary. We will generally be informal about the distinction between sets and classes. However, let us clarify from the outset that we adopt Tarski-Grothendieck set theory as the foundation for this work. Other alternatives exist (see, e.g., \cite[Sect.~1.7]{Fa19a}), but this issue is entirely beyond the scope of the paper.

We denote by $\mathbb N$ the non-negative integers, by $\mathbb N^+$ the positive integers, and by $\mathbb Z$ the integers. For all $a, b \in \allowbreak \mathbb Z$, we let $\llb a, b \rrb$ be the \evid{discrete interval} $\{x \in \mathbb Z\colon a \le x \le b\}$. 
We write $|X|$ for the cardinality of a set $X$. If $X_1, \ldots, X_n \subseteq \mathbb Z$, we define $X_1 + \cdots + X_n := \{x_1 + \cdots + x_n : x_1 \in X_1, \ldots, x_n \in X_n\}$.

\subsection{Monoids.}
\label{sec:2.2}
A \evid{monoid} is a semigroup with an identity element. Unless stated otherwise, monoids will typically be written multiplicatively and need not have any special property (e.g., commutativity). For the basics on semigroups and monoids, see \cite{Howie}.

Let $H$ be a monoid with identity $1_H$. An element $u \in H$ is a \evid{unit} if there is a provably unique $v \in H$, called the \evid{inverse} of $u$ and denoted by $u^{-1}$, such that $uv = 1_H = vu$.  We denote by $H^\times$ the set of units of $H$, which turns out to be a subgroup of~$H$ and is therefore named the \evid{group of units} of $H$. A non-unit $a \in H$ is an \evid{atom} if it does not factor as a product of two non-units (that is, $a=xy$ for some $x,y \in H$ implies $x \in H^\times$ or $y \in H^\times$). We denote by $\mathscr A (H)$ the set of atoms of $H$. 

We will consider many special classes of monoids that are relevant to our investigations, which makes it necessary to review some further terminology. Most notably, we say that $H$ is

\begin{itemize}[leftmargin=1cm]
\item \evid{reduced} if the only unit of $H$ is the identity, i.e., $H^\times = \{1_H\}$; 

\item \evid{acyclic} if $yxz\ne x$ for all $x, y, z \in H$ with $y \not\in H^\times$ or $z\not\in H^\times$;

\item \evid{unit-cancellative} if $xy \ne x$ and $yx \ne x$ for all $x, y\in H$ with $y\not\in H^\times$;

\item \evid{cancellative} if $xz \ne yz$ and $zx \ne zy$ for all $x, y, z \in H$ with $x \ne y$;

\item \evid{Dedekind-finite} if $xy = 1_H$ for some $x, y \in H$ implies $yx = 1_H$, or equivalently, if a non-empty product of elements of $H$ is a unit if and only if each factor is a unit \cite[Proposition 2.30]{Fan-Tri-2018}.
\end{itemize}
Acyclic or commutative monoids are Dedekind-finite, as are atomic monoids \cite[Lemma 2.2]{Fan-Tri-2018}.
A fundamental class of cancellative monoids is provided by the non-zero elements of a (commutative or non-commutative) domain under mul\-ti\-pli\-ca\-tion. Clearly, every cancellative monoid is unit-cancellative. The converse is not true; for instance, the non-empty finite subsets of $\mathbb Z$, endowed with the operation of setwise addition induced from the additive group of integers, form an atomic and unit-cancellative monoid that is not cancellative \cite[Proposition 3.5]{Fan-Tri-2018}.
It is also immediate that a commutative monoid is unit-cancellative if and only if it is acyclic. For a cancellative monoid that is not acyclic, see \cite[Example 4.8]{Tri-2022(a)}.

Given $X_1,\, \ldots,\, X_n \subseteq H$, we write $X_1 \cdots X_n$ for the \evid{setwise product} of $X_1$ through $X_n$, that is, the set $\{x_1 \cdots x_n \colon x_1 \in X_1,\, \ldots,\, x_n \in X_n\} \subseteq H$; if $X_i = \{x_i\}$ for some $i \in \llb 1, n \rrb$ and there is no likelihood of confusion, we will commonly replace the set $X_i$ in the product $X_1 \cdots X_n$ with the element $x_i$.

A \evid{monoid congruence} on $H$ is an equivalence relation $\Cal R$ on $H$ such that if $x \RR u$ and $y \RR v$ then $xy \RR uv$. If $\Cal R$ is a monoid congruence on $H$, we usually write $x \equiv y$ or $x \cong y$ in place of $x \RR y$ and say that ``$x$ is congruent to $y$ (modulo $\Cal R\,$)''.

\subsection{Free monoids}\label{sec:2.3} Fix a set $X$. We denote by $\mathscr F(X)$ the free monoid on $X$, that is, the monoid having as elements all finite tuples of elements of $X$ and whose operation is the \evid{con\-cat\-e\-na\-tion} of such tuples. We use the symbol $\ast_X$ for the operation of $\mathscr F(X)$ and we refer to an element of $\mathscr F(X)$ as an \evid{$X$-word}, or simply as a \evid{word} if no confusion can arise. The identity of $\mathscr F(X)$ is the empty tuple, called the \evid{empty $X$-word} and denoted by $\varepsilon_X$. Clearly, $\mathscr F(X)$ is a reduced atomic monoid with $\mathscr{A}(\mathscr F(X))=X$.

Note that if $\mathfrak u$ is a non-empty $X$-word, then $\mathfrak u = u_1 \ast_X \cdots \ast_X u_k$ for some uniquely de\-ter\-mined $k \in \mathbb N ^+$ and $u_1,\, \ldots,\, u_k \in X$. The positive integer $k$ is called the \evid{length} of $\mathfrak u$. By definition, the empty word has zero length. We denote the length of an $X$-word $\mathfrak u$ by $|u|_X$.
We will systematically drop the subscript $X$ from this notation when there is no serious risk of ambiguity.

\subsection{Factorizations and length sets}
\label{subsec:factorizations}
Let $H$ be a monoid with set of atoms $\mathscr A(H)$, and consider the (unique) morphism $\pi_H \colon \mathscr F(H)\to H$ from the free monoid on $H$ into $H$ that maps an $H$-word $\mathfrak a = a_1\ast \cdots \ast a_n$ to the element $a_1\cdots a_n \in H$. We define an \evid{atomic factorization} of an element $x \in H$ as an $\mathscr A(H)$-word $\mathfrak a$ such that $\pi_H(\mathfrak a)=x$. Then, we set 
$$
\mathcal Z_H(x):= \{\mathfrak a \in \mathscr F (\mathscr A(H)) : \mathfrak a \text{ is an atomic factorization of } x \}
$$
and 
$$
\mathsf L_H(x):= \{k \in \mathbb N \colon |\mathfrak a|_H = k, \text{ for some } \mathfrak a \in \mathcal Z_H(x)\}.
$$
We refer to $\mathsf L_H(x)$ as the \evid{length set} (or \evid{set of lengths}) of $x$ (relative to the atoms of $H$). Notice that $\Cal Z_H(1_H) = \allowbreak \{\varepsilon_{\mathscr A(H)}\}$ and $\Cal Z_H(u)=\emptyset$ for all $u \in H^\times \setminus \{1_H\}$, with the result that $\mathsf L_H(1_H) = \{0\}$ and $\mathsf L_H(u) = \allowbreak \emptyset$. Moreover, it is clear that $\mathsf L_H(x) = \{1\}$ if and only if $x \in \mathscr A(H)$. We also define
\[
\mathscr L(H) := \{ \mathsf L_H(x) \colon x \in H \} \setminus \{\emptyset\} 
\qquad\text{and}\qquad
\mathscr L^+(H) := \mathscr L(H) \setminus \{\{0\}\};
\]
we call $\mathscr L(H)$ and $\mathscr L^+(H)$ the \evid{system of length sets} and the \evid{system of non-zero length sets} of $H$, respectively.
Finally, given a non-negative integer $k$, we define the \evid{union of length sets of $H$ containing $k$} by
\begin{equation}\label{def:unions-of-length-sets}
\mathscr{U}_k(H) := \bigcup\,\{L \in \mathscr{L}(H) \colon k\in L\}.
\end{equation}

An elementary property of length sets that is useful to keep in mind for the remainder is that, if $H$ is a monoid and $x \in H$ is a non-unit, then $\mathsf{L}_H(x) = \mathsf{L}_H(uxv)$ for all $u, v\in H^\times$ (see Lemma 2.2(iv) in \cite{Fan-Tri-2018}).

\subsection{Categories}\label{sec:cats} We refer to \cite{Borceux} for generalities on category theory. In fact, we will only deal with very basic categorical notions: monomorphisms, functors and adjunctions, limits (products, pullbacks, and equalizers), and colimits (coproducts, pushouts, and coequalizers).  

In particular, we recall that a category is \evid{complete} (resp., \evid{cocomplete}) if it admits all small limits (resp., small colimits). It is well known that a category is complete if and only if it admits all small products (including the empty product, which coincides with the terminal object) and all equalizers of pairs of arrows \cite[Theorem~2.8.1]{Borceux}. The dual statement holds for cocomplete categories.

\section{The category of atomic monoids}
\label{sec:basics-on-atomon}

Throughout, we denote by $\textsf{Mon}$ the usual category of monoids and monoid homomorphisms, and we define $\AtoMon$ as the subcategory of $\textsf{Mon}$  whose objects are the \emph{atomic} monoids and whose morphisms are the atom-preserving homomorphisms between atomic monoids, where a monoid homomorphism $f \colon H \to K$ is \evid{atom-preserving} if $f(a) \in \mathscr{A}(K)$ for every $a \in \mathscr{A}(H)$. The first properties of the category $\AtoMon$ are collected in this section.

\begin{remark}\label{rem: AtoMon basics}
\begin{enumerate}[label=\textup{(\roman{*})}, wide, labelwidth=0pt, labelindent=0pt]
\item\label{rem: AtoMon basics i}
    We recall from Section~\ref{sec:2.2} that every atomic monoid $H$ is Dedekind-finite and this ensures that for every element $a \in H$, $a \in \mathscr{A}(H)$ if and only if $uav\in \mathscr{A}(H)$ for all $u, v\in H^\times$.

\item\label{rem: AtoMon basics ii}
   A morphism $f\colon H\to K$ in $\AtoMon$ sends non-units into non-units. Indeed, since $H$ is atomic, every non-unit $h$ of $H$ is a (non-empty) product of atoms, and so is $f(h)$, which is then a non-unit by Dedekind-finiteness. Therefore, for every element $h \in H$, $h \in H^\times$ if and only if $f(h)\in K^\times$. Moreover, if $h \in H \setminus (\mathscr A(H)\cup H^\times)$, then $f(h)\in K \setminus (\mathscr A(K) \cup K^\times)$. Hence, we also have $h \in \mathscr A(H)$ if and only if $f(h) \in \mathscr A(K)$. Therefore, a morphism in $\AtoMon$ sends units into units, atoms into atoms and non-units that are not atoms into non-units that are not atoms.

\item\label{rem: AtoMon basics v}
    The category $\AtoMon$ is not a full subcategory of $\Mon$. In fact, a morphism $\varphi \colon H \to K$ in $\Mon$ between two atomic monoids may fail to send atoms into atoms, hence it may fail to be a morphism in $\AtoMon$. If, for instance, $K$ is the monoid $(\mathbb N, +)$ of non-negative integers under addition and $H$ is a numerical monoid (i.e., a submonoid of $(\mathbb N, +)$ with finite complement in $\mathbb N$) such that $H \subsetneq K$, then the inclusion map $H \to K$ is a monomorphism in $\Mon$, but does not send atoms into atoms.
    
\item \label{rem: AtoMon basics vi}
A morphism $\varphi \colon H \to K$ in $\AtoMon$ is a monomorphism (in $\AtoMon$) if and only if its restriction to $\mathscr{A} (H)\cup H^\times$ is an injective map. It can be proved by a direct computation or by using Proposition~\ref{prop: pullbacks} and the fact that a morphism $\varphi \colon H \to K$ in a category $\Cal C$ is a monomorphism if and only if the square
    $$
    \xymatrix{
    H \ar[r]^-{\text{Id}_H}  \ar[d]_-{\text{Id}_H} & H \ar[d]^-\varphi \\
    H \ar[r]_-\varphi & K
    }
    $$
    is a pullback.
\item 
    We will show in Theorem~\ref{thm: cocomplete} that colimits in $\AtoMon$ are computed as in $\Mon$. Therefore, a morphism $\varphi\colon H \to K$ in $\AtoMon$ is an epimorphism (in $\AtoMon$) if and only if it is an epimorphism in $\Mon$.
    \end{enumerate}
\end{remark}

We identify the initial and terminal objects of the category $\AtoMon$.
\begin{proposition}\label{prop: initial and terminal}
 The initial object of $\AtoMon$ is the trivial monoid $\mathbb{0}:=\{1\}$, while its terminal object is the monoid $\mathbb{1}:=\{1, 0, a\}$ consisting of an identity $1$, an absorbing element $0$, and an element $a$ such that $a^2=0$. In particular, $\AtoMon$ does not have a zero object.
\end{proposition}
\begin{proof}
    We only need to prove that $\mathbb{1}:=\{1, 0, a\}$ is the terminal object of $\AtoMon$, the rest is trivial. It is clear from the definition that $\mathbb{1}^\times=\{1\}$ and that $\mathscr{A}(\mathbb{1})=\{a\}$. In particular, $\mathbb{1}$ is an atomic monoid. Moreover, for every $H \in \AtoMon$, there is a unique morphism $H\to \mathbb{1}$ sending all the units of $H$ to $1$, all atoms of $H$ in $a$ and all non-units of $H$ that are not atoms to $0$ (this assignment is well-defined by Dedekind-finiteness).
\end{proof}

\begin{remark}\label{rem: subobjects and submonoids}
    Notice that the notions of ``submonoid'' (in the usual sense) and ``subobject'' are not equivalent in the category $\AtoMon$. Recall that, for a category $\Cal C$, a subobject of a given object $X \in \Cal C$ is an equivalence class of monomorphisms $\alpha \colon A \to X$, where two monomorphisms $\alpha \colon A \to X$ and $\alpha' \colon A' \to X$ are equivalent if there is an isomorphism $\varphi\colon A \to A'$ such that $\alpha=\alpha' \varphi$. As is customary, if $\alpha\colon A \to X$ is a representative of a subobject of $X$, we shall simply say ``$A$ is a subobject of $X$'', omitting the morphism.
    
    Now, let $H,K \in \AtoMon$ with $H$ submonoid of $K$. Then, $H$ may fail to be a subobject of $K$ (that is, there might be no monomorphisms from $H$ to $K$). For instance, if $K$ is the monoid $(\mathbb N, +)$ of non-negative integers and $H$ is its submonoid $K \setminus \{1\}$, then $\mathscr A(K)=\{1\}$ and $\mathscr A(H)=\{2,3\}$, but by Remark~\ref{rem: AtoMon basics}(iv), there are no monomorphisms in $\AtoMon$ from $H$ to $K$ (for every morphism $\varphi:H \to K$ in $\AtoMon$, we have $\varphi(2)=\varphi(3)=1$).
\end{remark}

Since morphisms in $\AtoMon$ send atoms into atoms, the assignment $H \mapsto \mathscr{A}(H)$ defines a functor
\[
\mathscr{A}\colon \AtoMon \to \mathsf{Set},
\]
which is a subfunctor of the forgetful functor $U\colon \AtoMon \to \mathsf{Set}$ (that is, $\mathscr A (H) \subseteq U(H)$ for every $H \in \AtoMon$). We show in the next proposition that $\mathscr{A}$ admits a left adjoint.
\begin{proposition}\label{prop: adjunction}
     The free functor $\mathscr F \colon \mathsf{Set}\to \AtoMon$, sending a set $X$ to the free monoid $\mathscr F(X)$ over $X$, is left adjoint to $\mathscr{A}$. In particular, we have the following commutative diagram of functors, where $\mathscr F'$ is the free functor from $\mathsf{Set}$ to $\Mon$.
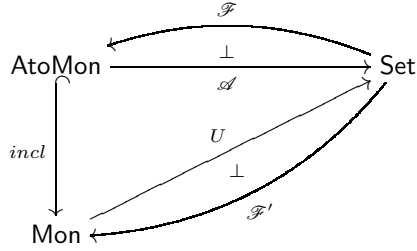
    $$
    \xymatrix@!=40pt{
    \AtoMon \ar[rr]_{\mathscr{A}}^\perp \ar@{^(->}[d]_{incl} & & \mathsf{Set} \ar@/_1.3pc/[ll]_{\mathscr F} \ar@/^1.5pc/[lld]^{\mathscr F'} \\
    \Mon \ar[rru]^{U}_[@][flip]\perp & & 
    }
    $$
    Moreover, for the adjunction $\mathscr{A} \vdash \mathscr F$, the $X$-components of the unit $\eta_X\colon X\to \mathscr{A}(\mathscr F (X))$ are the identity maps, while the $H$-components of the counit $\varepsilon_H \colon \mathscr F(\mathscr{A}(H)) \to H$ are given by the restriction of the factorization morphism $\pi_H\colon \mathscr{F}(H)\to H$.
\end{proposition}

\begin{proof}
    In order to show that $\mathscr F$ is left adjoint to $\mathscr{A}$, we want to prove that for any set $X$ and any atomic monoid $H \in \AtoMon$, there is an isomorphism $\hom_{\mathsf{Set}}(X, \mathscr{A} (H)) \cong \hom_\AtoMon(\mathscr F (X), H)$, which is natural both in $X$ and $H$. For any map $f \colon X \to \mathscr{A} (H)$, there is a unique extension to a monoid homomorphism $\hat{f}\colon \mathscr F(X)\to H$, which is clearly a morphism in $\AtoMon$, since the atoms of $\mathscr F(X)$ are precisely the elements of $X$. It is easy to check that this defines the required bijection, whose inverse assigns to any $g\colon \mathscr F(X)\to H$, the restriction $g|_X\colon X \to \mathscr{A}(H)$ (this is well-defined, since $g$ is a morphism in $\AtoMon$).
\end{proof}

\begin{remark}
    We will show in Example \ref{example: coproducts not preserved} that the functor $\mathscr{A}$ does not preserve colimits, hence it does not have a right adjoint (see \cite[Section~3]{Borceux} for a proof of the fact that functors admitting a right adjoint preserve colimits).
\end{remark}

\section{Coproducts}\label{sec: coproducts}
Throughout this section, $\mathcal H = \{H_i\}_{i \in I}$ is a family of atomic monoids indexed by a non-empty (possibly infinite) set $I$, and we denote by $\sqcup \, \mathcal H := \bigsqcup_{i \in I} H_i$ the disjoint union of (the underlying sets of) the monoids $H_i$ in the family $\mathcal H$. We aim to show that the coproduct of the family $\mathcal H$ in $\AtoMon$ is the same as the coproduct of $\mathcal H$ in the category $\Mon$.

We gather from \cite[Section~8.2]{Howie} that the coproduct of the monoids $H_i$ in $\Mon$ is, up to isomorphism, the \evid{free product} of the family $\mathcal H$, which is constructed as follows. We first consider the free monoid $\mathscr F(\sqcup\, \mathcal H)$ over the set $\sqcup \, \mathcal H$. Thus, an element in $\mathscr{F}(\sqcup\, \mathcal H)$ different from the empty word is of the form
    $$
    \mathfrak{a} :=\gword, \quad \text{ where } n \in \mathbb{N}^+,  i_j \in I \text{ and } a_{j}\in H_{i_j} \text{ for all } j\in \llb 1, n\rrb .
    $$
Then, we consider the smallest (semigroup) congruence $\cong$ on the free monoid $\mathscr F(\sqcup\, \mathcal H)$ such that, for all $i \in I$ and $a, b \in H_i$, the $\sqcup\, \mathcal H$-word $(a, i) \ast (b, i)$ is congruent to the $\sqcup\, \mathcal H$-word $(ab, i)$ and the $\sqcup\, \mathcal H$-word $(1_{H_i}, i)$ is congruent to the empty $\sqcup\, \mathcal H$-word $\varepsilon$. The free product $\HH$ is then the quotient of the free monoid $\mathscr F(\sqcup\, \mathcal H)$ by the congruence $\cong$ :
    $$
    \HH:=\FC /{\cong}  .
    $$
As mentioned above, $\HH$ is the (object-part of the) coproduct $\coprod_{i \in I} H_i$ of the family $\mathcal H$ in $\Mon$.
A non-empty word $\gword \in \FC$ is called {\em reduced} if $a_j \neq 1_{H_{i_j}}$ and $i_j \neq i_{j+1}$ for each $j \in \llb 1, n-1 \rrb$, and we let the empty word be reduced by definition. Every element of $\HH$ is uniquely represented by a reduced word. For a word $\mathfrak a \in \FC$, we will denote by $\widehat{\mathfrak a}$ the reduced word that represents the congruence class of $\mathfrak{a}$ in $\HH$ and refer to it as ``the reduced form of (the congruence class of) $\mathfrak a$ in $\HH$''.
In the following, we will often need to compare two (or more) congruent words in order to show some of their properties (e.g. if some of their letters are units or atoms), usually after suitable reductions of some of their parts. For this reason, it will be useful to point out the following results.

\begin{lemma}\label{lem: index-block_decomposition existence}
    Let $\mathfrak{a}=\gword \in \FC$ be a non-empty word.
    \begin{enumerate}[label=\textup(\arabic{*})]
        \item\label{lem: index-block_decomposition existence 1}
        If $\mathfrak a$ is congruent to the empty word, then $a_j \in H_{i_j}^\times$ for every $j\in \llb 1, n\rrb$.
        \item \label{lem: index-block_decomposition existence 2}
        If $\mathfrak a$ is not congruent to the empty word, then there exist $m \in \mathbb{N}^+$ and non-empty words $B_1, B_2, \dots, B_m \in \FC$ such that:
        \begin{enumerate}[label=\textup(\roman{*})]
            \item 
            $B_j\ncong \varepsilon$ for every $j \in  \llb 1, m\rrb$;
            \item
            $\mathfrak{a} \cong B_1\ast \cdots \ast B_m$;
            \item 
            for every $j\in \llb 1, m\rrb$ there exist $r_j\in \mathbb{N}^+$ and $\ell_j \in I$ such that $B_j=(b_{j,1}, \ell_j)\ast\cdots\ast (b_{j, r_j}, \ell_j)$;
            \item  for every $j\in \llb 1, m-1\rrb$, $\ell_j\neq \ell_{j+1}$.
        \end{enumerate}
    \end{enumerate}
\end{lemma}

\begin{proof}
    \ref{lem: index-block_decomposition existence 1} We proceed by induction. For $n=1$, we must have $(a_1,i_1)=(1_{H_{i_1}}, i_1)$ by the definition of the congruence. If $\mathfrak{a}=\gword$ with $n\geq 2$, since $\mathfrak a$ is equivalent to the empty word, two cases can occur.
    
    First case: there exists $\ell\in \llb 1, n\rrb$ such that $a_\ell=1_{H_{i_\ell}}$. Then, $\mathfrak{a}$ is equivalent to the word of length $n-1$ obtained from $\gword$ by replacing the letter $(a_\ell, i_\ell)$ with the empty word. Thus, the inductive hypothesis can be applied.
    
    Second case: there exists $\ell\in \llb 1, n-1\rrb$ such that $i_\ell=i_{\ell+1}$. Then, $\mathfrak{a}$ is equivalent to the word $(a_1, i_1)\ast \cdots \ast (a_\ell a_{\ell+1}, i_\ell) \ast \cdots \ast (a_n, i_n)$ of length $n-1$. By inductive hypothesis, $a_j \in H_{i_j}^\times$ for every $j\in \llb 1, n\rrb \setminus \{\ell\}$ and $a_\ell a_{\ell+1} \in H_{i_\ell}^\times$. Thus, also in this case the conclusion follows (recall that $H_{i_\ell}$ is atomic and hence Dedekind-finite). 

    \ref{lem: index-block_decomposition existence 2} It immediately follows from the existence of the reduced form of $\mathfrak{a}$. Indeed, if $(b_1, i_1)\ast \cdots \ast (b_m, i_m)$ is the reduced form of $\mathfrak{a}$, then it suffices to set $B_j=(b_j,i_j)$.
\end{proof}

\begin{definition}
    Given a word $\mathfrak{a} \in \FC$ not equivalent to the empty word and $B_1, B_2, \dots, B_m \in \FC$ as in Lemma \ref{lem: index-block_decomposition existence}\ref{lem: index-block_decomposition existence 2}, we shall call $B_1\ast B_2\ast \cdots \ast B_m$ an {\em index-block decomposition of $\mathfrak a$}.
\end{definition}
    
    Index-block decompositions are, of course, not unique. However, two index-block decompositions of the same word have the ``same length'' and the ``same blocks up to reduction'', as shown in the following lemma.

\begin{lemma} \label{lem: index-block_decomposition uniqueness}
    Let $\mathfrak{a}=\gword \in \FC$ be a word not congruent to the empty word. If $B_1\ast \cdots \ast B_m$ and $B'_1\ast \cdots B'_\ell$ are two index-block decompositions of $\mathfrak a$, then $m=\ell$ and $B_j\cong B'_j$ for every $j$. In particular, if $\mathfrak a$ is congruent to a single-letter word $(b, k)$, then $a_j \in H_{i_j}^\times$ for every $j\in \llb 1, n\rrb$ such that $i_j\ne k$;
\end{lemma}

\begin{proof}
    From the definition of the congruence $\cong$, we get that $\widehat{B_1\ast \cdots \ast B_n}=\widehat{B_1}\ast \cdots \ast \widehat{B_n}$. Thus, the result follows from the uniqueness of the reduced form. The last assertion follows from Lemma \ref{lem: index-block_decomposition existence}\ref{lem: index-block_decomposition existence 1}.
\end{proof}

In order to show that the coproduct of the family $\mathcal H$ in $\AtoMon$ is the same as the coproduct of $\mathcal H$ in the category $\Mon$, we first need to determine the units and the atoms of $\HH$.

\begin{lemma}\label{lem: units coproduct}
    The congruence class of an element
    $$
    \gword \in \FC
    $$
    is a unit in $\HH$ if and only if $a_j \in H_{i_j}^\times$ for every $j\in \llb 1, n \rrb$. In particular, $\HH$ is Dedekind-finite.
\end{lemma}

\begin{proof}
    The ``if-part'' is clear. For the other implication it suffices to apply Lemma~\ref{lem: index-block_decomposition existence}\ref{lem: index-block_decomposition existence 1}. The Dedekind-finiteness of $\HH$ is then a trivial consequence.
\end{proof}

\begin{lemma}\label{lem: atom coproduct}
    The equivalence class of an element
    $$
    \gword \in \FC
    $$
    is an atom in $\HH$ if and only if it can be represented by a word $\mathfrak u \ast (a, i)\ast \mathfrak v$, where $a$ is an atom of $H_i$ and $\mathfrak u$ and $\mathfrak v$ are representatives of units of $\HH$. In particular, $\HH$ is an atomic monoid.
\end{lemma}

\begin{proof}
Let $\mathfrak a=(a_1, i_1)\ast(a_2, i_2)\ast \cdots\ast(a_n,i_n)$ be a representative of an atom of $\HH$. Since $\HH$ is Dedekind-finite, by Remark~\ref{rem: AtoMon basics}\ref{rem: AtoMon basics i} we can assume without loss of generality that $a_1\in H_{i_1}\setminus H_{i_1}^\times$ and $a_n\in H_{i_n}\setminus H_{i_n}^\times$, which immediately implies that $n=1$, otherwise $[(a_1, i_1)][(a_2, i_2)\ast \cdots\ast(a_n,i_n)]$ would be a factorization of $[\mathfrak a]$ into two non-units. This implies at once that $\mathfrak a=(a_1, i_1)$ with $a_1 \in \mathscr{A}(H_{i_1})$.

On the other hand, let $a$ be an atom of $H_i$. We claim that $(a,i)$ represents an atom of $\HH$. For, assume by contradiction that $\mathfrak{b}=(b_1, i_1)\ast \cdots \ast (b_n,i_n)$ and $\mathfrak{c}=(c_1, j_1)\ast \cdots \ast (c_m,j_m)$ are the reduced representations of two elements of $\HH\setminus \HH^\times$ such that $(a,i)\cong \mathfrak{b}\ast \mathfrak{c}$ (i.e., $[(a,i)]=[\mathfrak{b}][\mathfrak{c}]$). Notice that, since $[\mathfrak{b}], [\mathfrak{c}] \notin \HH^\times$, there exist $k \in \llb 1, n \rrb$ and $\ell \in \llb 1,m \rrb$ such that $b_k \notin H_{i_k}^\times$ and $c_\ell \notin H_{j_\ell}^\times$. By the uniqueness of the reduced form, we must have $i_k = j_\ell= i$ and the word $\mathfrak{b}\ast \mathfrak{c}$ reduces to a word of the form $(x, i)\ast (y,i)\cong (xy,i)$, where $x$ and $y$ are non-units of $H_i$. This implies $a=xy$ in $H_i$, which is a contradiction.

It remains to prove that $\HH$ is an atomic monoid, i.e., that every non-unit of $\HH$ is a product of elements of $\mathscr{A}(\HH):=\{[\mathfrak u][(a,i)][\mathfrak v] \colon a\in \mathscr{A}(H_i) \text{ and } [\mathfrak u],[\mathfrak v]\in \HH^\times\}$. Let $\mathfrak a=\gword$ be a representative of a non-unit of $\HH$. By Lemma~\ref{lem: units coproduct}, there is at least one $j\in \llb 1, n\rrb$ such that $a_j$ is a non-unit of $H_{i_j}$. For each of such indexes, being $H_{i_j}$ atomic, $[(a_j, i_j)]=[(x_1, i_j)]\cdots [(x_\ell, i_j)]$ for some $x_1, \dots, x_\ell \in \mathscr{A}(H_{i_j})$. The claim follows then trivially. 
\end{proof}

We are finally ready to prove the main theorem of this section.

\begin{theorem}\label{thm:non-empty-coproducts}
 Let $\mathcal{H} = \{H_i\}_{i \in I}$ be a (possibly infinite) family of monoids in $\AtoMon$. The coproduct $(\coprod_{i \in I} H_i, \{e_i : i \in I\})$ of $\mathcal{H}$ in the category $\Mon$ is also the coproduct of $\mathcal{H}$ in $\AtoMon$.
\end{theorem}
\begin{proof}
 We proved in Proposition \ref{prop: initial and terminal} that the initial object in $\AtoMon$ is the same as the one in $\Mon$, so we can assume $I \neq \emptyset$. Set $\HH:=\coprod_{i \in I} H_i$. First observe that Lemma~\ref{lem: atom coproduct} ensures that $\HH\in \AtoMon$, i.e., $\HH$ is an atomic monoid. Also, it is clear that the coproduct coprojections $e_i\colon H_i \to \HH$, which send an element $a \in H_i$ to the element $[(a, i)] \in \HH$, are morphisms in $\AtoMon$. Now, given a family $\{\varphi_i\colon H_i \to K\}_{i \in I}$ of morphisms in $\AtoMon$, the unique morphism $\sigma\colon \HH \to K $ induced by the coproduct in $\Mon$ is defined by the assignment $[\gword] \mapsto \varphi_{i_1}(a_1)\cdots \varphi_{i_n}(a_n)$. So, taking into account the characterization of atoms of $\HH$, we get that $\sigma$ sends atoms of $\HH$ into atoms of $K$ because all the $\varphi_i$'s preserve atoms and units. Thus, $\sigma$ is a morphism in $\AtoMon$.
\end{proof}

The following example shows that the functor $\mathscr{A} \colon \AtoMon \to \mathsf{Set}$ considered in Proposition \ref{prop: adjunction} does not preserve (finite) coproducts, hence it can not have a right adjoint.

\begin{example}\label{example: coproducts not preserved} 
    Let $H_1$ be an infinite group and $H_2$ any finite atomic monoid that is not a group (hence with at least one atom). Clearly $\mathscr{A}(H_1)= \emptyset$ and $\mathscr{A}(H_2)=\mathscr{A} (H_1)\sqcup \mathscr{A}(H_2)$ is finite. On the other hand, the atoms of $H_1 \coprod H_2$ are all elements of the form $[(g_1, 1)][(a, 2)](g_2, 1)]$ with $g_1,g_2 \in H_1$ and $a \in \mathscr{A}(H_2)$ by Lemma \ref{lem: atom coproduct}. Thus, $\mathscr{A}(H_1 \coprod H_2)$ is an infinite set and therefore the functor $\mathscr{A}\colon \AtoMon \to \mathsf{Set}$ does not preserve (binary) coproducts.
\end{example}

The coproduct of a family $\mathcal{H}=\{H_i\}_{i \in I}$ of atomic monoids inherits their cancellativity properties.
\begin{proposition}\label{prop: preserved properties coproduct}
    Let {\bf P} be one of the following properties:
\begin{enumerate}[label=\textup{(\arabic{*})}, mode=unboxed]
    \item\label{prop: preserved properties coproduct 1} being acyclic;
    \item\label{prop: preserved properties coproduct 2} being unit-cancellative;
    \item\label{prop: preserved properties coproduct 3} being cancellative.
\end{enumerate}
If $\mathcal{H}=\{H_i\}_{i \in I}$ is a family of monoids in $\AtoMon$ satisfying property {\bf P}, then $\HH$ satisfies {\bf P}. 
\end{proposition}

\begin{proof}
    
    \ref{prop: preserved properties coproduct 1}  Let $\mathcal{H}=\{H_i\}_{i \in I}$ be a family of acyclic monoids in $\AtoMon$, and let $\mathfrak a, \mathfrak u, \mathfrak v\in \FC$ be the reduced representatives of some elements $[\mathfrak a],[\mathfrak u],[\mathfrak v]\in \HH$ such that $[\mathfrak a]=[\mathfrak u][\mathfrak a][\mathfrak v]$, i.e., $\mathfrak a \cong \mathfrak u\ast \mathfrak a \ast \mathfrak v$. If $\mathfrak a$ is empty, then $[\mathfrak{u}]$ and $[\mathfrak{v}]$ are clearly units of $\HH$ by Dedekind-finiteness (see Lemma~\ref{lem: units coproduct}).
    
    Now, assume that $\mathfrak a=\gword$ is non-empty. We first observe that, without loss of generality, we can assume $a_1\in H_{i_1}\setminus H_{i_1}^\times$. Indeed, if this is not the case, we can replace $\mathfrak{a}$ with $\mathfrak{a}':=(a_2, i_2)\ast \cdots \ast (a_n, i_n)$ and $\mathfrak{u}$ with (the reduced form of) $\mathfrak{u}':=(a_1^{-1}, i_1)\ast \mathfrak{u}\ast (a_1,i_1)$. We clearly have $[\mathfrak{a}']=[\widehat{\mathfrak{u}}'][\mathfrak{a}'][\mathfrak{v}]$ and $\mathfrak{u}$ is a unit of $\HH$ if and only if so is $\widehat{\mathfrak{u}}'$. Symmetrically, we can also assume that $a_n\in H_{i_n}\setminus H_{i_n}^\times$.
    If $\mathfrak{u}=(u_1, j_1)\ast \cdots \ast (u_m,j_m)$ is a non-empty word, we can reduce the word $\mathfrak{u}\ast\mathfrak{a}$ only if $i_1=j_m$, replacing $(u_m,j_m)\ast (a_1,i_1)$ with $(u_m a_1,i_1)$, and no further reductions are possible since $a_1\in H_{i_1}\setminus H_{i_1}^\times$ and both $\mathfrak a$ and $\mathfrak u$ are reduced. Arguing in the same way for $\mathfrak{a}\ast\mathfrak{v}$, we then get from Lemma~\ref{lem: index-block_decomposition uniqueness} that the congruence $\mathfrak a \cong \mathfrak u\ast \mathfrak a \ast \mathfrak v$ can only take place if either $\mathfrak u=\varepsilon$ or $\mathfrak u=(u, i_1)$ and either $\mathfrak v=\varepsilon$ or $\mathfrak v=(v, i_n)$. Therefore $a_1=ua_1v$ in $H_{i_1}$ if $n=1$, or $a_1=u a_1$ in $H_{i_1}$ and $a_n=a_n v$ in $H_{i_n}$ if $n \geq 2$ (here $u$ and $v$ can possibly be the identities of their monoids). The acyclicity of $H_{i_1}$ and $H_{i_n}$ implies that $u$ and $v$ are units in the monoids they belong to, and so $[\mathfrak{u}]=[(u, i_1)]$ and $[\mathfrak{v}]=[(v, i_n)]$ are units of $\HH$.

    \ref{prop: preserved properties coproduct 2} We can argue as in the previous point, fixing either  $\mathfrak{u}=\varepsilon$ or $\mathfrak{v}=\varepsilon$.
    
    \ref{prop: preserved properties coproduct 3} Let $\mathfrak a, \mathfrak b, \mathfrak c\in \FC$ be the reduced representatives of some elements $[\mathfrak a],[\mathfrak b], [\mathfrak c]\in \HH$ such that $[\mathfrak a][\mathfrak b]=[\mathfrak a][\mathfrak c]$. If $\mathfrak a=\varepsilon$, then $[\mathfrak b]=[\mathfrak c]$ trivially. Hence, we can assume that $\mathfrak a$ is non-empty. Moreover, if one between $\mathfrak b$ or $\mathfrak c$ is empty, we get by \ref{prop: preserved properties coproduct 2} that they both are. We can then assume that $\mathfrak a=\gword$, $\mathfrak b=(b_1,j_1)\ast\cdots\ast(b_r,j_r)$ and $\mathfrak c=(c_1,k_1)\ast \cdots\ast (c_s,k_s)$ are non-empty reduced words in $\FC$ such that $\mathfrak a\ast \mathfrak b\cong \mathfrak a \ast \mathfrak c$. By an argument similar to the one used in the proof of \ref{prop: preserved properties coproduct 1}, we can also assume that $a_n \in H_{i_n}\setminus H_{i_n}^\times$. Thus, we can reduce the word $\mathfrak{a}\ast\mathfrak{b}$ [resp., $\mathfrak{a}\ast\mathfrak{c}$] only if $i_n=j_1$ [resp., $i_n=k_1$] by replacing $(a_n,i_n)\ast (b_1,j_1)$ with $(a_n b_1,i_n)$ [resp., $(a_n,i_n)\ast (c_1,k_1)$ with $(a_n c_1,i_n)$], and no further reductions are possible since $a_n$ is a non-unit of $H_{i_n}$ and $\mathfrak{a,b,c}$ are reduced words. In particular, since the reduced forms of $\mathfrak a \ast \mathfrak b$ and $\mathfrak a \ast \mathfrak c$ are equal, we get from Lemma \ref{lem: index-block_decomposition uniqueness} that the equality $(a_n,i_n) \ast (b_1,j_1) \ast \dots \ast (b_r,j_r)=(a_n,i_n)\ast (c_1,k_1)\ast \dots \ast (c_s,k_s)$ holds in $\FC$. Notice that neither $j_1=i_n \neq k_1$ nor $j_1\neq i_n = k_1$ can happen. Indeed, if $j_1=i_n \neq k_1$, again by Lemma \ref{lem: index-block_decomposition uniqueness}, we get $a_n b_1=a_n$ and the cancellativity of $H_{i_n}$ implies $b_1=1_{H_{j_1}}$, which is a contradiction since $\mathfrak b$ is a reduced word. Similarly for the case $j_1\neq i_n = k_1$.
    Thus, we can either have $j_1= i_n = k_1$ or $j_1\neq i_n \neq k_1$. In both cases, the uniqueness of the reduced form and the cancellativity of $H_{i_n}$ immediately imply that $\mathfrak b = \mathfrak c$. This concludes the proof since the argument for the case $[\mathfrak b][\mathfrak a]=[\mathfrak c][\mathfrak a]$ is analogous.
\end{proof}

We conclude this section with a series of arithmetic results that illustrate how the length sets of the coproduct of a family of monoids in $\AtoMon$ and the related arithmetic invariants are connected to those of the monoids in the family. The reader may find useful to revise Section~\ref{subsec:factorizations} before proceeding.
 
\begin{proposition}\label{prop: s.o.l. coproduct}
  Let $\HH$ be the coproduct of a non-empty family $\mathcal{H} = \{H_i\}_{i \in I}$ of monoids in $\AtoMon$. Let $\mathfrak a=\gword\in \FC$ be the reduced representative of a non-unit $[\mathfrak a]\in \HH$, and let $J$ be the set of all $j \in \llb 1, n\rrb$ such that $a_j\not\in H_{i_j}^\times$. Then, 
    \[\mathsf{L}_{\HH}([\mathfrak a])=\sum_{j\in J} \mathsf{L}_{H_{i_j}}(a_j).\]  
\end{proposition}

\begin{proof}
We already proved in Lemma~\ref{lem: atom coproduct} that $\HH$ is atomic so $\mathsf{L}_{\HH}([\mathfrak a])$ is a non-empty subset of the positive integers. 
Let $m\in \mathsf{L}_{\HH}([\mathfrak a])$. Then, taking into account the characterization of the atoms of $\HH$ given in Lemma~\ref{lem: atom coproduct}, \[\mathfrak a=\gword \cong \mathfrak u_1 \ast (w_1, j_1) \ast \mathfrak v_1 \ast \cdots \ast \mathfrak u_m \ast (w_m, j_m) \ast \mathfrak v_m,\] where for every $k\in \llb 1, m\rrb$, $j_k\in\{i_1,\dots,i_n\}$, $\mathfrak u_k \ast (w_k, j_k) \ast \mathfrak v_k$ is a reduced word, $\mathfrak u_k, \mathfrak v_k$ are (reduced) representatives of units of $\HH$, and $w_k \in \mathscr A(H_{j_k})$. It then follows from Lemma~\ref{lem: index-block_decomposition uniqueness} (so, ultimately, from the uniqueness of the reduced form), that for every $j\in \llb1,n\rrb$, $a_j=\alpha_0^{(j)}\alpha^{(j)}_1\cdots \alpha^{(j)}_{s_j}$, where $s_j \in \mathbb N$, $\alpha_0^{(j)}\in H_{i_j}^\times$, and $\alpha^{(j)}_\ell \in \mathscr{A}(H_{i_j})$ is associated (in $H_{i_j}$) to one of the atoms $w_1,\dots, w_m$ for every $\ell \in \llb 1, s_j\rrb$ (note that $s_j=0$ for every $j\not \in J$). Moreover, it is clear that $m=\sum_{j\in J}s_j\in \sum_{j\in J}\mathsf{L}_{H_{i_j}}(a_j)$.

On the other hand, let $m\in \sum_{j\in J} \mathsf{L}_{H_{i_j}}(a_j)$. Then, for every $j\in J$, there exists an atomic factorization $\mathfrak z_j =z_1^{(j)}\ast \cdots \ast z_{n_j}^{(j)}\in \Cal Z_{H_{i_j}}(a_j)$ of length $n_j:=|\mathfrak z_j|\in \mathbb N^+$, and $m=\sum_{j\in J}n_j$.
If we identify the word $\mathfrak z_j$ with $(z_1^{(j)}, i_j)\ast \cdots \ast (z_{n_j}^{(j)}, i_j) \in \FC$, it is then clear that $\mathfrak a$ is congruent to a word
$$
\mathfrak z:= {\ast}_{j\in J} (\mathfrak{u}_j\ast \mathfrak{z}_j \ast \mathfrak{v}_j ),
$$
where for every $j \in J$, $\mathfrak u_j$ and $\mathfrak v_j$ are reduced representatives of units of $\HH$ (possibly empty). Moreover, it follows from Lemma~\ref{lem: atom coproduct} that for every $j\in J$, $\mathfrak u_j\ast (z_1^{(j)}, i_j)$ and $(z_{n_j}^{(j)}, i_j)\ast \mathfrak v_j$ are representatives of atoms of $\HH$, as well as $(z_{\ell_j}^{(j)}, i_j)$ for every $\ell_j\in \llb 1,n_j\rrb$. 
Therefore, from $\mathfrak z$ we get an atomic factorization of $[\mathfrak a]$ in $\HH$ of length $m$ and so $m\in \mathsf L_{\HH}([\mathfrak a])$, as claimed.
\end{proof}

\begin{theorem}\label{thm: s.s.o.l. coproduct}
Let $\HH$ be the coproduct of a non-empty family $\mathcal{H}=\{H_i\}_{i \in I}$ of monoids in $\AtoMon$. Denote by $\Gamma_\mathcal{H}$ the following subset of the free monoid $\mathscr F(I)$ over $I$:
\begin{enumerate}
\item If there exist $i,j \in I$ with $i \ne j$ such that $H_i^\times$ and $H_j^\times$ are both non-trivial, then $\Gamma_\mathcal{H} := \mathscr F(I)$.
\item If there is a unique $i \in I$ such that $H_i^\times$ is non-trivial, then $\Gamma_\mathcal{H}$ is the set of all $I$-words without consecutive letters both equal to $i$.
\item Otherwise, 
$\Gamma_\mathcal{H}$ is the set of all non-empty $I$-words any two consecutive letters of which are different.
\end{enumerate}

Then, the system of non-zero length sets of $\HH$ is given by
\begin{equation}
\label{thm:s.s.o.l.coproduct:equ(1)}
\mathscr{L}^+(\HH)=
\bigcup_{i_1 \ast \cdots \ast i_n \in \Gamma_\mathcal{H}} \{L_1 + \cdots + L_n: L_1 \in \mathscr{L}^+(H_{i_1}), \ldots, L_n \in \mathscr{L}^+(H_{i_n})\} 
\end{equation}
\end{theorem}

\begin{proof}
Denote by $\mathscr L$ the set on the right-hand side of Eq.~\eqref{thm:s.s.o.l.coproduct:equ(1)}, and let $L\in \mathscr{L}^+(\HH)$. By definition, there exists a non-unit $x\in \HH$ such that $L=\mathsf L_{\HH}(x)$. 
Let $\mathfrak a= \gword\in \FC$ be the reduced representative of $x$, and let $\mathfrak a'$ be the ordered subword  of $\mathfrak a$ obtained by removing from $\mathfrak a$ all the letters of the form $(a_j, i_j)$ with $a_j \in H_{i_j}^\times$. Thus, we can write $\mathfrak a'=(a_{j_1}, i_{j_1})\ast \cdots \ast (a_{j_m}, i_{j_m})$ for a suitable subset $J:=\{j_1, \dots, j_m\}$ of $\llb 1, n \rrb$. Notice that $J$ is non-empty since $x$ is a non-unit of $\HH$, and it is trivial but tedious to check that the word $\mathfrak i = i_{j_1}\ast \cdots \ast i_{j_m} \in \mathscr F(I)$ is actually in $\Gamma_{\mathcal H}$. It then follows from Proposition~\ref{prop: s.o.l. coproduct} that $\mathsf{L}_{\HH}(x)=\sum_{k=1}^{m} \mathsf{L}_{H_{\mathfrak i_{j_k}}}(a_{j_k})$ and hence $L\in \mathscr L$.

Conversely, let $L \in \mathscr L$. There then exist a non-empty word $\mathfrak i=i_1 \ast \cdots \ast i_n \in \Gamma_\mathcal{H}$ and, for each $j \in \llb 1, n \rrb$, a non-unit $a_j \in H_{i_j}$ such that $L = \mathsf{L}_{H_{i_1}}(a_1) + \cdots + \mathsf{L}_{H_{i_n}}(a_n)$. We need to find a reduced word $\mathfrak b \in \FC$ such that $\mathsf{L}_{\HH}([\mathfrak b])=L$. To this end, consider the word $\mathfrak a := \gword \in \FC$. If $\mathfrak a$ is already reduced, set $\mathfrak b:= \mathfrak a$. Otherwise, there is an index $j \in \llb 1, n-1\rrb$ such that $i_j=i_{j+1}$ (note that this can only happen in cases (1) or (2) of the statement). By hypothesis, since $\mathfrak i$ is in $\Gamma_{\mathcal H}$, there is $k\in I$ such that $k\ne i_j$ and $H_{k}$ is not reduced. Thus, there exists $b\in H_k^\times \setminus \{1_{H_k}\}$ and we can replace the subword $(a_j, i_j)\ast (a_{j+1}, i_{j+1})$ of $\mathfrak a$ with the reduced word $(a_j, i_j)\ast(b, k)\ast (a_{j+1}, i_{j+1})$. Repeating this procedure for all $j \in \llb 1, n-1\rrb$, we obtain the desired reduced word $\mathfrak b$. 
\end{proof}

\begin{theorem}\label{thm: unions of s.o.l. coproduct}
  Let $\HH$ be the coproduct of a non-empty family $\mathcal{H}=\{H_i\}_{i \in I}$ of monoids in $\AtoMon$, and let $\Gamma_\mathcal H$ be as in the previous theorem. Then, for any positive integer $k$,  
    \[\mathscr{U}_k(\HH)=\bigcup_{i_1 \ast \cdots \ast i_n \in \Gamma_\mathcal{H}}
    \{\mathscr{U}_{k_1}(H_{\mathfrak i_1}) + \cdots + \mathscr{U}_{k_n}(H_{\mathfrak i_n}): k_1, \ldots, k_n \in \mathbb N^+
    \text{ and } k_1 + \cdots + k_n = k\}
    \]
\end{theorem}

\begin{proof}
Let $m\in \mathscr{U}_k(\HH)$. Then $m\ne 0$ and, by definition of union and  Theorem \ref{thm: s.s.o.l. coproduct} there exists a non-empty $\mathfrak i=i_1 \ast \cdots \ast i_n \in \Gamma_\mathcal H$ and $L=L_1 + \cdots + L_n$, with $L_j\in \mathscr{L}^+(H_{i_j})$ for every $j\in \llb 1, n\rrb$, such that $\{m,k\}\subseteq L$. Then, for every $j\in \llb 1, n\rrb$, there exist $m_j, k_j\in L_j$ (both positive) such that $\sum_{j=1}^n m_j=m$ and $\sum_{j=1}^n k_j=k$. Thus, $m_j\in \mathscr{U}_{k_j}(H_{\mathfrak i_j})$ for every $j$, and so $m\in \sum_{j=1}^{n}\mathscr{U}_{k_j}(H_{i_j})$, as desired.

On the other hand, assume that $m \in \sum_{j=1}^{n}\mathscr{U}_{k_j}(H_{i_j})$, where $i_1 \ast \cdots \ast i_n\in \Gamma_\mathcal H$,  $k_1, \ldots, k_n \in \mathbb N^+$, and $\sum_{j=1}^n k_j=k$. Then, for every $j\in \llb 1, n\rrb$, there exists $L_j\in \mathscr{L}^+(H_{i_j})$ containing $k_j$ such that $m \in \sum_{j=1}^n L_j$. Moreover, it is clear that also $k\in \sum_{j=1}^n L_j$. It then follows from Theorem~\ref{thm: s.s.o.l. coproduct}, that $\{m,k\}\in L:=\sum_{j=1}^nL_j \in \mathscr{L}^+(\HH)$, so $m\in \mathscr{U}_k(\HH)$.
\end{proof}

\section{Products}\label{sec: products}
In this section we want to show the existence of arbitrary products in $\AtoMon$. We will see that, contrary to the case of coproducts, products in $\AtoMon$ are different from those in $\Mon$. We shall use the term ``direct product'' and the standard symbols $\prod$ and $\times$ whenever we refer to products in $\Mon$, while we shall specify ``product in $\AtoMon$'' and introduce suitable symbols whenever required by the context.

\begin{remark} Let $\mathcal H = \{H_i\}_{i \in I}$ be a family of atomic monoids indexed by a non-empty (possibly infinite) set $I$. We first want to show that the direct product $\prod_{i \in I}H_i$ does not coincide with the (categorical) product in $\AtoMon$. The problem is two-fold: on one hand the direct product of atomic monoids may fail to be atomic and, on the other hand, even if the direct product is an atomic monoid, the canonical projections may fail to be morphisms in $\AtoMon$ (actually this is always the case, unless the product consists of a unique factor). It is routine to check that an element $f \in \prod_{i \in I} H_i$ is an atom if and only if there is $j \in I$ such that $f(j) \in \mathscr A(H_j)$ and $f(i) \in H^\times_i$ for every $i \neq j$. From this, we already get that the canonical projections are not atom-preserving, unless $\mathcal H$ consists of a single monoid. It is also easy to see that if the family of monoids is finite, then the direct product is an atomic monoid. Nevertheless, this is not always the case for arbitrary products. Let assume that $\mathcal H = \{H_i\}_{i \in \mathbb N}$ is a family of atomic monoids such that for each $k \in \mathbb N$ there is an element $x_k \in H_k$ of minimal length factorization equal to $k$ (i.e., $\min \mathsf{L}(x_k)=k$). Then, the element $f \in \prod_{i \in \mathbb N} H_i$ defined by $f(i)=x_i$ for every $i \in \mathbb N$ cannot be written as a finite product of atoms, hence $\prod_{i \in \mathbb N} H_i \notin \AtoMon$.
\end{remark}

We now compute the product in $\AtoMon$ of a family of atomic monoids. Let $\mathcal H = \{H_i\}_{i \in I}$ be a family of atomic monoids indexed by a non-empty (possibly infinite) set $I$. Let us consider the following two subsets of the direct product $\prod_{i\in I}H_i$:
$$
\UPA:=\{f\in \prod_{i\in I}H_i \colon f(i)\in H_i^\times \text{ for every } i\in I\}
$$
and
$$
\APA:=\{f\in \prod_{i\in I}H_i \colon f(i)\in \mathscr{A}(H_i)\text{ for every } i\in I\},
$$
Then, define the monoid $\PA$ as the submonoid of the direct product $\prod_{i\in I}H_i$ generated by $\UPA \cup \APA$. 

\begin{remark}\label{rem: atoms}
    It is clear from Remark~\ref{rem: AtoMon basics}\ref{rem: AtoMon basics i} that $\UPA\APA\UPA=\APA$.
\end{remark}

We will show that $\PA$ is the (object-part of the) product of the family $\mathcal{H}$ in $\AtoMon$, but we first characterize units and atoms of $\PA$ and prove that it is an atomic monoid.
\begin{lemma}\label{lem: units and atoms}
    Let $f$ be an element in the direct product $\prod_{i\in I}H_i$. Then:

\begin{enumerate*}[label=\textup{(\roman{*})}, mode=unboxed]
        \item\label{lem: units and atoms 1} $f$ is a unit of $\PA$ if and only if $f\in\UPA$, so $\PA$ is Dedekind-finite;
\end{enumerate*}

\begin{enumerate*}[label=\textup{(\roman{*})}, mode=unboxed, resume]
        \item\label{lem: units and atoms 2} $f$ is a non-unit of $\PA$ if and only if it is a non-empty finite product of elements of $\APA$;
\end{enumerate*}

\begin{enumerate*}[label=\textup{(\roman{*})}, mode=unboxed, resume]
        \item\label{lem: units and atoms 3} $f$ is an atom of $\PA$ if and only if $f\in \APA$. In particular, the monoid $\PA$ is atomic.
\end{enumerate*}
\end{lemma}

\begin{proof}
\ref{lem: units and atoms 1} It is clear that the elements of $\UPA$ are units of $\PA$. On the other hand, if $f$ is a unit of $\PA$, its $i$-th component $f(i)$ is a unit of $H_i$ for every $i\in I$, so $f \in \UPA$.

\ref{lem: units and atoms 2} Every element of $\PA$ is a non-empty finite product of its generators. For $f\in \PA$, we then have $f=f_1 \cdots f_m$, with $f_j\in \UPA\cup \APA$ for every $j\in \llb 1, m\rrb$. If $f$ is not a unit of $\PA$, by item \ref{lem: units and atoms 1}, there is a subset $J\subseteq\llb 1, m\rrb$ with cardinality $n\ge 1$, such that $f_j\in \APA$ for every $j\in J$, and $f_k \in \UPA$ for every $k\in \llb 1,m\rrb \setminus J$. We then obtain from Remark~\ref{rem: atoms} that $f$ is a product of $n$ elements of $\APA$. On the other hand, assume $f=f_1\cdots f_n$ with $n\ge 1$ and $f_1,\ldots, f_n$ in $\APA$. Then $f$ is clearly a non-unit of $\PA$, again by item \ref{lem: units and atoms 1}.

\ref{lem: units and atoms 3} The claim follows immediately from item \ref{lem: units and atoms 2}.
\end{proof}

\begin{theorem}\label{thm:non-empty-products}
 Let $\mathcal{H}=\{H_i\}_{i \in I}$ be a family of monoids in $\AtoMon$ indexed by a non-empty (possibly infinite) set $I$. The product of $\mathcal{H}$ in $\AtoMon$ is given by $(\PA, \{\pi_i : i \in I\})$, where $\pi_i \colon \PA\to H_i$ is the restriction of the canonical projection from the direct product $\prod_{i \in I}H_i$ into $H_i$ for every $i \in I$.
\end{theorem}

\begin{proof}
 First observe that Lemma~\ref{lem: units and atoms} ensures that $\PA\in \AtoMon$, and that the monoid homomorphisms $\pi_i\colon \PA \to H_i$ sending an element $f \in \PA$ to the element $f(i) \in H_i$, are morphisms in $\AtoMon$. Now, given a family $\{\varphi_i\colon H \to H_i\}_{i \in I}$ of morphisms in $\AtoMon$, we define $\sigma\colon H \to \PA $ by the assignment $x \mapsto \sigma(x)$, where $\sigma(x)(i)=\varphi_i(x)$ for every $i\in I$. Let us first verify that $\sigma(x)$ is an element of $\PA$. For $x\in H^\times$, $\varphi_i(x)\in H_i^\times$ for every $i\in I$ so $\sigma(x)\in \UPA\subseteq \PA$. Now, let $x \in H \setminus H^\times$. Since $H$ is atomic, there exist $a_1,\dots, a_n \in \mathscr{A}(H)$ such that $x=a_1\cdots a_n$. Then $\varphi_i(x)=\varphi_i(a_1)\cdots\varphi_i(a_n)$ for every $i\in I$, and $\varphi_i(a_j)\in \mathscr{A}(H_i)$ for every $i\in I$ and for every $j\in \llb 1,n\rrb$. Now, for every $j\in \llb 1,n\rrb$, define $\sigma_j\in \prod_{i\in I}H_i$ by the assignment $\sigma_j(i)=\varphi_i(a_j)$. It is clear that each $\sigma_j\in \APA$, and so $\sigma(x)=\sigma_1\cdots \sigma_n\in \PA$, as desired. It is also routine to check that $\sigma$ is a morphism of $\AtoMon$ and that $\phi_i=\pi_i \cdot \sigma$ for every $i\in I$. Finally, it is clear that $\sigma$ is the unique morphism of $\AtoMon$ satisfying this property.
\end{proof}

Just like the coproduct, the product $\PA$ of $\AtoMon$ also preserves the cancellative properties of the monoids in $\mathcal{H}$.

\begin{proposition}\label{prop: preserved_properties_prod}
    Let {\bf P} be one of the following properties:
\begin{enumerate}[label=\textup{(\arabic{*})}, mode=unboxed]
    \item\label{prop: preserved properties product 1} being acyclic;
    \item\label{prop: preserved properties product 2} being unit-cancellative;
    \item\label{prop: preserved properties product 3} being cancellative.
\end{enumerate}
If $\mathcal{H}=\{H_i\}_{i \in I}$ is a family of monoids in $\AtoMon$ satisfying property {\bf P}, then $\PA$ satisfies {\bf P}. 

\end{proposition}
\begin{proof}
    The proof is immediate arguing component-wise.
\end{proof}

In the rest of the section we provide a nice description of some arithmetical invariants of $\PA$ in terms of the arithmetical invariants of the elements of $\mathcal{H}$.

\begin{proposition}\label{prop: s.o.l. product}
    Let $\PA$ be the product of the family $\mathcal{H}=\{H_i\}_{i \in I}$ in $\AtoMon$ and let $f$ be an element of $\PA$. Then 
    \[\mathsf{L}_{\PA}(f)=\bigcap_{i\in I} \mathsf{L}_{H_i}(f(i)).\]
\end{proposition}
\begin{proof}
    If $f=1_{\PA}$, then $f(i)=1_{H_i}$ for every $i\in I$, so $\mathsf{L}_{\PA}(f)=\mathsf{L}_{H_i}(f(i))=\{0\}$ for every $i\in I$, and we are done. If $f\in \PA^\times \setminus \{1_{\PA}\}$, then $f(i)\in H_i^\times$ for every $i\in I$ and there is at least one $j\in I$ such that $f(j)\ne 1_{H_{j}}$, so $\mathsf{L}_{\PA}(f)=\mathsf{L}_{H_j}(f(j))=\bigcap_{i\in I} \mathsf{L}_{H_i}(f(i))=\emptyset$. Assume then that $f$ is a non-unit of $\PA$ and that $n\in \mathsf L_{\PA}(f)$. By Lemma~\ref{lem: units and atoms}\ref{lem: units and atoms 3}, there exist $\alpha_1, \dots, \alpha_n\in \APA$ such that $f=\alpha_1\cdots\alpha_n$. Then, for every $i\in I$, $f(i)=\alpha_1(i)\cdots \alpha_n(i)$, and since $\alpha_k(i)\in \mathscr{A}(H_i)$ for every $k\in \llb 1, n\rrb$, $n\in \mathsf{L}_{H_i}(f(i))$ for every $i$. This shows that $\mathsf{L}_{\PA}(f)\subseteq \bigcap_{i\in I} \mathsf{L}_{H_i}(f(i))$. On the other hand, assume $n\in \mathsf{L}_{H_i}(f(i))$ for every $i\in I$. This means that for every $i\in I$, $f(i)=a^{(i)}_1\cdots a^{(i)}_n$ with $a^{(i)}_k\in \mathscr{A}(H_i)$ for every $k\in \llb 1, n\rrb$. For each $k\in \llb 1,n\rrb$, define then $\alpha_k\in \PA$ by $\alpha_k(i)=a^{(i)}_k$. By Lemma~\ref{lem: units and atoms}\ref{lem: units and atoms 3} $\alpha_1, \dots, \alpha_n$ are atoms of $\PA$, and $f=\alpha_1\cdots\alpha_n$. Thus $n\in \mathsf{L}_{\PA}(f)$ and this concludes the proof.
\end{proof}

\begin{theorem}\label{thm: s.s.o.l. product}
    Let $\PA$ be the product of a family $\mathcal{H}=\{H_i\}_{i \in I}$ of monoids in $\AtoMon$. Then
\begin{equation*}
\label{equ:system-of-length-sets-of-prods}
    \mathscr{L}(\PA)=\bigsqcap_{i\in I}\mathscr{L}(H_i) := \left\{\bigcap_{i\in I} L_i \colon L_i\in \mathscr{L}(H_i) \text{ for every }i\in I\right\}\setminus \{\emptyset\}.
\end{equation*}   
\end{theorem}
\begin{proof}
    If $L\in \mathscr{L}(\PA)$ then $L=\mathsf{L}_{\PA}(f)$ for some $f\in \PA$. By Proposition~\ref{prop: s.o.l. product}, $L=\bigcap_{i\in I} \mathsf{L}_{H_i}(f(i))$, so $L \in \bigsqcap_{i\in I}\mathscr{L}(H_i)$. (Note that since $L$ is non-empty, so are the $\mathsf{L}_{H_i}(f(i))$ for every $i\in I$.)
    
    On the other hand, let $L \in \bigsqcap_{i\in I}\mathscr{L}(H_i)$. If $0\in L$, then $L=\{0\}=\mathsf{L}_{\PA}(1_{\PA})\in \mathscr{L}(\PA)$. Assume then that $0\not\in L$. By assumption $\emptyset\ne L=\bigcap_{i\in I}L_i$, where $L_i\in \mathscr{L}(H_i)$, so there exist elements $x_i\in H_i$ such that $L=\bigcap_{i\in I}\mathsf{L}_{H_i}(x_i)$. For every $i\in I$ and every positive integer $n\in L$ (there is at least one), $x_i=a^{(i)}_1\cdots a^{(i)}_n$, where $a^{(i)}_1, \dots, a^{(i)}_n\in \mathscr{A}(H_i)$. Define $f: I\to \bigcup_{i\in I}H_i$ by $f(i):=x_i$. Then $f$ is an element of $\PA$ (it is a product of elements of $\APA$), and  Proposition~\ref{prop: s.o.l. product} implies that $L=\bigcap_{i\in I}\mathsf{L}_{H_i}(x_i)=\mathsf{L}_{\PA}(f)\in \mathscr{L}(\PA)$.
\end{proof}

\begin{theorem}\label{thm: unions product}
    Let $\PA$ be the product of a family $\mathcal{H}=\{H_i\}_{i \in I}$ of monoids in $\AtoMon$. Then
    \[\mathscr{U}_k(\PA)=\bigcap_{i\in I}\mathscr{U}_k(H_i),
    \qquad\text{for every } i \in I.
    \] 
\end{theorem}

\begin{proof}
Fix $k \in \mathbb N$ and set $\Omega_k := \bigcap_{i \in I} \mathscr{U}_k(H_i)$. We need to show that $n \in \mathscr U_k(\PA)$ if and only if $n \in \Omega_k$. To begin, we have from Theorem \ref{thm: s.s.o.l. product} and Eq.~\eqref{def:unions-of-length-sets} in Section \ref{subsec:factorizations} that

\begin{equation}
\label{equ:unions-of-length-sets-in-prods}
\mathscr{U}_k(\PA)=\bigcup\left\{L\in \bigsqcap_{i \in I}\mathscr{L}(H_i) \colon k\in L\right\}.
\end{equation}
If $n \in \mathscr{U}_k(\PA)$, then Eq.~\eqref{equ:unions-of-length-sets-in-prods} guarantees the existence of a sequence $(L_i)_{i \in I}$ of subsets of $\mathbb N$ such that $n \in \bigcap_{i \in I} L_i$ and $k \in L_i \in \mathscr{L}(H_i)$. It follows that $n \in \mathscr{U}_k(H_i)$ for all $i \in I$, and hence $n \in \Omega_k$.

Conversely, if $n \in \Omega_k$,
then for every $i\in I$ there is a set $L_i\in \mathscr{L}(H_i)$ such that $\{n,k\}\subseteq L_i$, with the result that $\{n,k\}\subseteq \bigcap_{i\in I} L_i\in \bigsqcap_{i\in I}\mathscr{L}(H_i)$ and hence $n\in \mathscr{U}_k(\PA)$.
\end{proof}

\section{Other limits and colimits}\label{sec: limits and colimits}
In this section we prove that the category $\AtoMon$ is both complete and cocomplete, that is, it admits all limits and colimits. We start by coequalizers and pushouts, showing that they coincide with those in $\Mon$. In particular, all colimits are computed as in $\Mon$. Then, we prove that the category $\AtoMon$ admits equalizers and pullbacks, hence all limits. Contrary to the case of colimits, limits in $\AtoMon$ are different from those in $\Mon$ (as already noticed for the terminal object and products), and we provide an explicit description for equalizers and pullbacks.

\subsection{Other colimits: coequalizers and pushouts}

\begin{proposition}[Coequalizers] \label{prop: coequalizers}
    Let $f,g\colon H \rightrightarrows K$ be a pair of parallel arrows in $\AtoMon$. Then, the coequalizer of $f$ and $g$ in $\AtoMon$ coincides with the coequalizer of $f$ and $g$ in $\Mon$.
\end{proposition}

\begin{proof}
    The object-part of the coequalizer of $f$ and $g$ in $\Mon$ is the quotient $Q:=K/{\sim}$, where $\sim$ is the smallest congruence on $K$ containing the pairs $(f(h), g(h))$ for all $h \in H$ (see \cite[Section~8.2]{Howie}). Let $x \sim y$ in $K$. Then, there exist $n \in \mathbb{N}$, elements $x_0, x_1, \dots, x_n, a_1, \dots, a_n, b_1, \dots, b_n \in K$ and $h_1, \dots h_n \in H$ such that $x_0=x, x_n=y$ and $\{  x_{i-1}, x_i  \} = \{ a_if(h_i)b_i, a_ig(h_i)b_i \}$ for every $i \in \llb 1, n-1\rrb$. We distinguish two cases.
    
    \textsc{Case 1:} $x \in K^\times$. From the Dedekind-finiteness and Remark \ref{rem: AtoMon basics}\ref{rem: AtoMon basics ii}, we get $a_1, b_1, f(h_1), g(h_1) \in K^\times$ (note that $h_1\in H^\times$), hence $x_1 \in K^\times$. By iterating this process, we can then conclude that $y \in K^\times$. Thus, if $[u]\in Q^\times$, there then exists $v \in K$ such that $1\sim uv$, hence $u, v \in K^\times$. This shows that $[x]\in Q^\times$ if and only if $x \in K^\times$, and in particular $Q$ is Dedekind-finite.

    \textsc{Case 2:} $x \in \mathscr{A}(K)$. Set $x=x_0=a_1 f(h_1)b_1$ and $x_1=a_1g(h_1)b_1$. Only two cases can occur. If $a_1f(h_1) \in K^\times$ and $b_1 \in \mathscr{A}(K)$, arguing as in Case 1, we get that $a_1 g(h_1)\in K^\times$ and so $x_1$ is an atom. On the other hand, if $b_1 \in K^\times$ and $a_1f(h_1) \in \mathscr{A}(K)$, we need to distinguish two subcases depending on whether $f(h_1)$ is a unit or an atom of $K$. If $f(h_1)\in K^\times$, then $a_1 \in \mathscr{A}(K)$ and the same argument as before allow us to conclude that $g(h_1) \in K^\times$ and $x_1$ is an atom. If $f(h_1) \in \mathscr{A}(K)$ ( and so $a_1 \in K^\times$), then $h_1 \in \mathscr{A}(H)$ by Remark \ref{rem: AtoMon basics}\ref{rem: AtoMon basics ii}, hence $g(h_1) \in \mathscr{A}(K)$, and so $x_1 \in \mathscr{A}(K)$.
    By iterating this process, we can then conclude that $y \in \mathscr{A}(K)$. This implies that if $x \in \mathscr{A}(K)$, then $[x] \in \mathscr{A}(Q)$. Indeed, if $[x]=[y][z]$ for some $y,z \in K$, then $yz \in \mathscr{A}(K)$, forcing $[y]$ or $[z]$ to be in $Q^\times$ by the characterization of invertible elements in $Q$ we have done before.

    Moreover, it is immediate to check that if $x$ is a non-unit, non-atom of $K$, then $[x]$ is not an atom in $Q$. This shows at once that $Q$ is an atomic monoid and that the canonical projection $\pi \colon K \to Q$ is a morphism in $\AtoMon$.
    \end{proof}

\begin{proposition}[Pushouts] \label{prop: pushouts}
    Let $f\colon L \to H$ and $g\colon L \to K$ be two morphisms in $\AtoMon$. Then, the pushout of $f$ and $g$ in $\AtoMon$ coincides with the pushout of $f$ and $g$ in $\Mon$.
\end{proposition}

\begin{proof}
    Recall that in any category $\Cal C$ with coequalizers and (binary) coproducts, the pushout of two arrows $f\colon L \to H$ and $g\colon L \to K$ can be constructed as follows (see  \cite[Section~2]{Borceux}): we first consider the coproduct of $H$ and $K$
    $$
    \xymatrix{H \ar[r]^-{e_H} & H \coprod K & K \ar[l]_-{e_K}}
    $$
    and then the coequalizer $q\colon H\coprod K \to Q$ of the parallel morphisms $e_H f$ and $e_K g$. Then, the commutative square
    $$
    \xymatrix{
    L \ar[r]^f  \ar[d]_g & H \ar[d]^{q e_H} \\
    K \ar[r]_-{q e_K} & Q
    }
    $$
    is a pushout in $\Cal C$. Since coproducts and coequalizers in $\AtoMon$ exist and coincide with those in $\Mon$, the conclusion follows.
\end{proof}

We showed in Theorem \ref{thm:non-empty-coproducts} and Proposition \ref{prop: coequalizers} that arbitrary coproducts and coequalizers exist in $\AtoMon$. Thus, the next result follows from the dual of \cite[Theorem~2.8.1]{Borceux}.

\begin{theorem} \label{thm: cocomplete}
    The category $\AtoMon$ is cocomplete, that is, all colimits exist in $\AtoMon$. Moreover, colimits are computed as in $\Mon$.
\end{theorem}

\subsection{Other limits: equalizers and pullbacks}

\begin{proposition}[Equalizers]\label{prop: equalizers}
    Let $f,g\colon H \rightrightarrows K$ be a pair of parallel arrows in $\AtoMon$. We define
    $$
    \mathscr{A} _{f,g} (H):=\{ x \in \mathscr{A}(H)\mid f(x)=g(x) \} \quad \text{ and } \quad  H^\times_{f,g}:=\{ x \in H^\times : f(x)=g(x) \}.
    $$
    Then, the equalizer of $f$ and $g$ in $\AtoMon$ exists and is given by the diagram
    $$
    \xymatrix{
    E(f,g) \ar@{>->}[r]^-{e} & H \ar@<0.6ex>[r]^{f} \ar@<-0.6ex>[r]_{g} & K
    }
    $$
    where $E(f,g)$ is the submonoid of $H$ generated by $\mathscr{A} _{f,g}(H)$ and $H^\times_{f,g}$, and $e$ is the inclusion map.
\end{proposition}

\begin{proof}
    The submonoid $E(f,g)$ of $H$ generated by $\mathscr{A} _{f,g}(H)$ and $H^\times_{f,g}$ is clearly atomic (the atoms of $E(f,g)$ are precisely the elements in $\mathscr{A}_{f,g}(H)$), and the inclusion $e \colon E(f,g)\hookrightarrow H$ is a morphism in $\AtoMon$ equalizing $f$ and $g$.
    
    Now, if $\alpha\colon W \to H$ is a morphism in $\AtoMon$ such that $f\alpha=g\alpha$, then we must have $\alpha(\mathscr{A}(W))\subseteq \mathscr{A} _{f,g}(H)$ and $\alpha(W^\times) \subseteq  H^\times_{f,g}$, which implies that $\alpha$ factors through $e$. The uniqueness of this factorization is clear, since $e$ is a monomorphism.
\end{proof}

Pullbacks in $\AtoMon$ can be explicitly described with the same idea (notice that the existence of pullbacks in $\AtoMon$ is already guaranteed by the existence of products and equalizers \cite[Section~2]{Borceux}).

\begin{proposition}[Pullbacks]\label{prop: pullbacks}
    Let $f\colon H \to L$ and $g\colon K \to L$ be two morphisms in $\AtoMon$. Then, the pullback of $f$ and $g$ in $\AtoMon$ is given by the diagram
    $$
    \xymatrix{
    P \pullbackcorner \ar[r]^-{p_2} \ar[d]_-{p_1} & K \ar[d]^g\\
    H \ar[r]^f & L
    }
    $$
    where $P$ is the submonoid of the direct product $H \times K$ generated by 
    $$
    \mathscr{A}_P :=\{ (x, y) \in \mathscr{A}(H)\times \mathscr{A}(K) : f(x)=g(y)  \} \quad \text{ and } \quad U_P :=\{ (x,y) \in H^\times \times K^\times : f(x)=g(x) \},
    $$
    and $p_1, p_2$ are the restrictions of the canonical projections.
\end{proposition}

\begin{proof}
    The submonoid $P$ of the direct product $H \times K$ generated by $\mathscr{A}_P$ and $U_P$ is clearly atomic, and $p_1, p_2$ are morphisms in $\AtoMon$, being the atoms of $P$ precisely the elements in $\mathscr{A}_P$. Moreover, it is clear that $f p_1=g p_2$.
    
    Now, let $\alpha\colon W\to H$ and $\beta\colon W \to K$ be two morphisms in $\AtoMon$ such that $f\alpha=g\beta$. Since $\alpha(\mathscr{A}(W))\subseteq \mathscr{A}(H)$ and $\beta(\mathscr{A}(W)) \subseteq \mathscr{A}(K)$, the map $\omega\colon W \to P$ defined by $w \mapsto (\alpha(w), \beta(w))$ is a morphism in $\AtoMon$ such that $\alpha=p_1 \omega$ and $\beta= p_2 \omega$. The uniqueness of $\omega$ is clear.
\end{proof}

\begin{theorem} \label{thm: complete}
    The category $\AtoMon$ is complete.
\end{theorem}

\begin{proof}
We know from Theorem \ref{thm:non-empty-products} and Proposition \ref{prop: equalizers} that $\AtoMon$ has arbitrary products and equalizers. Thus, we conclude from \cite[Theorem~2.8.1]{Borceux} that it has all limits, namely, it is complete.
\end{proof}

\section{Prospects for future research}\label{sec: open problems}

The results in this paper reveal intriguing properties of the category $\AtoMon$ of atomic monoids and atom-preserving monoid homomorphisms, highlighting its potential to advance factorization theory from a category-theoretic perspective. While our findings mark only a first step in this direction, they provide a foundation for further research and suggest promising developments. A key question concerns the role of Theorems \ref{thm: s.s.o.l. coproduct} and Theorems \ref{thm: s.s.o.l. product} in realizing systems of length sets. Another problem is to obtain a (sensible) categorical characterization of transfer homomorphisms \cite[Definition 2.1]{Ba-Sm15}, a special class of monoid homomorphisms that are central to the classical theory of factorization. 

A possible extension is to consider factorizations where the ``building blocks'' are irreducibles in the sense of \cite{Cos-Tri-2025, Cos-2025} rather than atoms. Such an idea leads to the introduction of the category 
$\FaMon$ of factorable monoids and irreducible-preserving monoid homomorphisms, where a monoid $H$ is \evid{factorable} if every non-unit is a product of irreducibles, and an \evid{irreducible} of $H$ is a non-unit $a \in H$ such that $a \ne xy$ for all \textit{proper} divisors $x,y \in H$ of $a$ (that is, if $a = xy$, then $a \mid_H x$ or $a \mid_H y$). This requires an analysis of limits and colimits, as well the study of arithmetic properties of universal constructions in $\FaMon$.

A further generalization shifts focus to \evid{premonoids}, that is, structures of the form $(H, \preceq)$ where $H$ is a monoid and $\preceq$ is a preorder (i.e., a reflexive and transitive relation) on $H$. The goal is to build a ``category of factorable premonoids'', where factorizations arise in the broader framework of \cite{Tri-2022(a), Cos-Tri-2021, Cos-Tri-2023, Cos-Tri-2024}. One problem is that, in this setting, it is not quite clear what the ``right morphisms'' should be.

\section*{Acknowledgments}

F.~Campanini is a postdoctoral researcher of the Fonds de la Recherche Scientifique (FNRS).
L.~Cossu acknowledges support by the European Union's Horizon 2020 program (Marie Sklodowska-Curie grant 101021791), the Austrian Science Fund FWF (grant DOI 10.55776/P\-AT\-975\-6623), and the project ``FIATLUCS'' funded by the PNRR RAISE Liguria, Spoke 01, CUP: F23C24000240006; she is a member of the National Group for Algebraic and Geometric Structures and their Applications (GNSAGA), a department of the Italian Mathematics Research Institute (INdAM). S.~Tringali was supported by grant A2023205045, funded by the Natural Science Foun\-da\-tion of Hebei Province.

\end{document}